\documentclass[11pt,a4paper]{amsart}
\usepackage{a4}
\usepackage{enumerate}
\usepackage{amssymb}
\usepackage{amsthm}
\usepackage{exscale}
\newcommand{\RR}{\mathbb{R}}
\newcommand{\CC}{\mathbb{C}}

\newcommand{\NN}{\mathbb{N}}	
\newcommand{\ZZ}{\mathbb{Z}}

\newcommand{\EE}{\, \mathbb{E} \,}
\newcommand{\PP}{\mathbb{P}}

\newcommand{\Tr}{{\mathop{\mathrm{Tr} \,}}}

\newcommand{\rank}{\mathop{\mathrm{rank}}}

\newcommand{\wj}{w^{(j)}}
\newcommand{\wjj}{w^{(j-1)}}

\newcommand{\supp}{{\mathop{\mathrm{supp\,}}}}
\newcommand{\be}{\begin{equation}}
\newcommand{\ee}{\end{equation}}

\newtheorem{thm}{Theorem}
\newtheorem{lem}[thm]{Lemma}

\newtheorem{cor}[thm]{Corollary}
\theoremstyle{definition}

\theoremstyle{remark}
\newtheorem{rem}[thm]{Remark}

\newcommand{\hm}[1]{\leavevmode{\marginpar{\tiny%
$\hbox to 0mm{\hspace*{-0.5mm}$\leftarrow$\hss}%
\vcenter{\vrule depth 0.1mm height 0.1mm width \the\marginparwidth}%
\hbox to
0mm{\hss$\rightarrow$\hspace*{-0.5mm}}$\\\relax\raggedright #1}}}

\begin{document}

\title[Wegner estimate for discrete alloy-type models]
{Wegner estimate for \\ discrete alloy-type models}

\author
{Ivan Veseli\'c}
\address{ Fakult\"at f\"ur Mathematik,\, 09107\, TU-Chemnitz, Germany  }
\urladdr{www.tu-chemnitz.de/mathematik/stochastik}

\thanks{
{\today, \jobname.tex}}


\begin{abstract}
We study discrete alloy-type random Schr\"odinger operators on $\ell^2(\ZZ^d)$.
Wegner estimates are bounds on the average number of eigenvalues in an energy interval of finite box restrictions of
these types of operators.
If the single site potential is compactly supported and the distribution of the coupling constant is of bounded variation
a Wegner estimate holds.  The bound is polynomial in the volume of the box and thus
applicable as an ingredient for a localisation proof via multiscale analysis.
\end{abstract}

\maketitle
\section{Main results}

A discrete alloy-type model is a family of operators $H_\omega=H_0+V_\omega$
on $\ell^2(\ZZ^d)$.  Here $H_0$ denotes an arbitrary symmetric operator. In most applications $H_0$ is the discrete Laplacian on $\ZZ^d$.
The random part $V_\omega$ is a multiplication operator
\begin{equation} \label{eq:alloy}
V_\omega (x) =\sum_{k \in \ZZ^d} \omega_k \, u(x-k)
\end{equation}
defined in terms of an i.~i.~d.{} sequence $\omega_k \colon \Omega \to \RR, k \in \ZZ^d$
of random variables each having a density $f$, and a single site potential $u \in \ell^1(\ZZ^d;\RR)$.
It follows that the mean value $\bar u := \sum_{k\in\ZZ^d} u(k)$ is well defined.
We will assume throughout the paper that $u$ does not vanish identically and that
 $f\in BV$. Here $BV$ denotes the space of functions with bounded total variation and
$\|\cdot\|_{BV}$ denotes the corresponding norm. The mathematical expectation {w.r.t.}~ the product measure associated
with the random variables $\omega_k,k \in \ZZ^d$ will be denoted by $\EE$.

The estimates we want to prove do not concern the operator $H_\omega, \omega\in  \Omega$
but rather its finite box restrictions. Thus for the purposes of the present paper domain and selfadjointness properties of
$H_\omega$ are irrelevant.
For $L\in \NN$ we denote the subset
$[0,L]^d \cap \ZZ^d$ by $\Lambda_L$, its characteristic function by $
\chi_{\Lambda_L} $,
the canonical inclusion $\ell^2(\Lambda_L)\to
\ell^2(\ZZ^d)$ by $\iota_L$ and the adjoint restriction $\ell^2(\ZZ^d)\to
\ell^2(\Lambda_L)$ by $\pi_L$. The finite cube restriction of $H_\omega$
is then defined as $H_{\omega,L} := \pi_L H_0 \iota_L + V_\omega \chi_{\Lambda_L}\colon \ell^2(\Lambda_L)\to \ell^2(\Lambda_L)$. For any $\omega \in \Omega$  and $L\in \NN$ the
restriction $H_{\omega,L}$ is a selfadjoint finite rank operator.
In particular its spectrum consists entirely of real eigenvalues $E(\omega,L,1)\le E(\omega,L,n)\le \dots \le E(\omega,L,\sharp \Lambda_L) $
counted including multiplicities.
Note that if $u$ has compact support, then there exists an $n \in \NN$ and an $x\in\ZZ^d$
such that $\supp u \subset \Lambda_{-n}+x$, where $ \Lambda_{-n}:= \{ -k\mid k \in \Lambda_n\}$.
We may assume without loss of generality $x=0$ without restricting the model \eqref{eq:alloy}.
The number of points in the support of $u$ is denoted by $\rank u$.
Now we are in the position to state our bounds on the expected number of eigenvalues
of finite box Hamiltonians $H_{\omega,L}$ in a compact energy interval $[E-\epsilon, E + \epsilon]$	.

\begin{thm} \label{thm:degenerate}
Assume that the single site potential $u$ has support in $\Lambda_{-n}$.
Then there exists a constant $c_u$ depending only on $u$ such that for any
$L\in \NN$, $E \in \RR$ and $\epsilon>0$ we have
\[
\EE \left \{\Tr \big [\chi_{[E-\epsilon,E+\epsilon]}(H_{\omega,L})\big]\right\}
\le
c_u \,  \rank u \,   \|f\|_{BV}   \ \epsilon\, (L+n)^{d\cdot (n+1)}
\]
\end{thm}
\begin{rem}
\label{rem:thm}
\begin{enumerate}
\item By the assumption on the support of the single site potential
$\rank u \le (n+1)^d$
\item
The constant $c_u$ is given in terms of derivatives of a finite array of polynomials constructed in terms of values of the function $u$.
 \item A bound of the type as it is given in Theorem \ref{thm:degenerate}
is called Wegner estimate. If such a bound holds one is interested in the dependence
of the RHS on the \emph{length of the energy interval} (in our case $2 \epsilon$) and
on the \emph{volume of the cube} $\Lambda_L$ (in our case $L^d$).
More precisely,  a general Wegner estimate is be of the form 
\[
\forall\, L\in \NN, E \in \RR,\epsilon>0 \colon
\EE \left \{\Tr \big [\chi_{[E-\epsilon,E+\epsilon]}(H_{\omega,L})\big]\right\}
\le
constant \ (2\epsilon)^a\, (L^d)^{b}
\]
 with some $a\leq1$ and $b\geq 1$. The best possible estimate is obtained in the case
$a=1$ and $b=1$.  Such a bound is, for instance, encountered in Corollary \ref{cor:dos} below.
\item
Our bound is linear in the energy-interval length and polynomial in the volume of the cube.
This implies that the Wegner bound can be used for a localisation proof via multiscale analysis,
see e.g.{} \cite{FroehlichS-83,DreifusK-89,Kirsch-08}.
More precisely, if an appropriate initial scale estimate is available,
the multiscale analysis --- using as an ingredient the Wegner estimate as given in Theorem
\ref{thm:degenerate}  ---  yields Anderson localisation. As the Wegner bound is valid on the whole energy axis
one can prove Anderson localisation in any energy region where the initial scale estimate holds. 
\item
One might ask whether the exponent $d\cdot(n+1)$ of the length scale is optimal for 
the model under consideration. To give an answer to this question one has to be more precise:
It seems that this exponent is the best one can obtain using a conventional scheme of proof which 
at its heart only uses local averaging over one random variable. There are more elaborate techniques, 
used e.g  in the proof of a Wegner estimate for an mutidimensional model with Bernoulli disorder \cite{BourgainK-05}
where averaging over local families of random variables gives estimates which are impossible to obtain using just 
wiggling a single parameter. Such techniques could yield a better volume dependence than the one in 
Theorem~\ref{thm:degenerate}.
\item 
At the end of the paper we discuss how to derive spectral and exponential localisation 
in the large disorder regime with the help of Theorem~\ref{thm:degenerate}.
\item
If the single site potential $u$ does not have compact support, one has to use an enhanced version of the
multiscale analysis  and so-called uniform Wegner estimates to prove localisation, see \cite{KirschSS-98b}.
However, there exist criteria which allow one to turn a standard Wegner estimate into a uniform one,
see, e.g., Lemma 4.10.2 in \cite{Veselic-07b}.
\item
The main point of the theorem is that no assumption on $u$ (apart from the
compact support) is required. In particular, the sign of $u$ can change arbitrarily.
The single site potential may be even degenerate in the sense that $ \bar u=0$.
Also, note that the result holds on the whole energy axis.
These two properties are in contrast to earlier results on Wegner estimates
for sign-changing single site potentials. See the discussion of the previous literature at the end of this section.
\item
If $u$ does satisfy the assumption $\bar u\neq 0$ we obtain an even better bound.
This is the content of Theorem \ref{thm:non-degenerate} below.
\end{enumerate}
\end{rem}
%
%

The next Theorem applies to single site potentials $u\in \ell^1(\ZZ^d)$
with non vanishing mean $\bar u \neq 0$.
Let $m \in \NN $ be such that $\sum_{\|k\|\ge m} |u(k)| \le |\bar u/2|$. Here  $\|k\|=\|k\|_\infty$
denotes the sup-norm.
\begin{thm} \label{thm:non-degenerate}
Assume $\bar u \neq 0$ and that $ f$ has compact support. Then we have
for any $L\in \NN$, $E \in \RR$ and $\epsilon>0$
\[
\EE \left \{\Tr \big [\chi_{[E-\epsilon,E+\epsilon]}(H_{\omega,L})\big]\right\}
\le
\frac{8}{\bar u}  \,  \min\big(L^d,\rank u\big)  \, \|f\|_{BV}  \ \epsilon\, (L+m)^{d}
\]
\end{thm}
In the case that the support of $u$ is compact,
we have an important
\begin{cor}
\label{cor:dos}
Assume $\bar u \neq 0$ and $ \supp u \subset \Lambda_{-n}$.
Then we have for any $L\in \NN$, $E \in \RR$ and $\epsilon>0$
\[
\EE \left \{\Tr \big [\chi_{[E-\epsilon,E+\epsilon]}(H_{\omega,L})\big]\right\}
\le
\frac{4}{\bar u}  \, \rank u \,   \|f\|_{BV} \ \epsilon\, (L+n)^{d}
\]
In particular, the function  $\RR \ni E \to \EE \left \{\Tr \big [\chi_{(-\infty,E]}(H_{\omega,L})\big]\right\}$
is  Lipschitz continuous.
\end{cor}
If the operator $H_\omega$ has a well defined integrated density of states $N\colon \RR \to \RR$, meaning that
\[
\lim_{L\to \infty} \frac{1}{L^d}
\EE \left \{\Tr \big [\chi_{(-\infty,E]}(H_{\omega,L})\big]\right\}
= N(E)
\]
at all continuity points of $N$, then Corollary \ref{cor:dos} implies that
the integrated density of states is Lipschitz continuous. Consequently its derivative,
the density of states, exists for almost all $E \in \RR$.

\begin{rem}
The situation that the two cases $ \bar u \neq0$ and
$ \bar u =0$ have to be distinguished occurs also in other contexts, see for instance the paper \cite{Klopp-02c} on weak disorder localisation.
\end{rem}

When looking at Theorems \ref{thm:degenerate} and \ref{thm:non-degenerate}
one might wonder what kind of Wegner bound holds for non-compactly supported single site potentials with vanishing mean.
To apply the methods of the present paper in this case
it seems that one has to require that $u$ tends to zero exponentially fast. 
In this situation one can hope to treat the decaying potential as a sufficiently small perturbation of a compactly supported potential.
So far only the case of one space dimension is settled:

\begin{thm} \label{thm:degenerate-exponential}
Assume that $f$ has compact support and that
there exists $s\in (0,1)$ and $C \in (0,\infty)$ such that $|u(k)| \le C s^{|k|}$
for all $ k \in \ZZ$. Then there exist $c_u \in (0,\infty)$ and
$D \in \NN_0$  depending only on $u$
such that for each $ \beta > D/|\log s|$ there exists a constant $ K_\beta \in (0,\infty)$
such that for all $L\in \NN$, $ E \in \RR$ and $\epsilon >0$
\[
\EE \left \{\Tr \big [\chi_{[E-\epsilon,E+\epsilon]}(H_{\omega,L})\big]\right\}
\le
\frac{8}{c_u}  \, \|f\|_{BV} \ \epsilon\, L \, (L+\beta \log L +K_\beta)^{D+1}
\]
\end{thm}

Let us discuss the relation of the above theorems to
previous results \cite{Klopp-95a,Veselic-02a,HislopK-02,KostrykinV-06,Veselic-circulant}
on Wegner estimates with single site potentials which are allowed to change sign.
The papers \cite{Klopp-95a,HislopK-02}
concern alloy-type Schr\"odinger operators on $L^2(\RR^d)$.
The main result is a Wegner estimate for energies in a neighbourhood
of the infimum of the spectrum. It applies to arbitrary non-vanishing
single site potentials $u \in C_c(\RR^d)$ and coupling constants with 	a picewise
absolutely continuous density. The upper bound is linear in the volume of the box
and H\"older-continuous in the energy variable. This means in the notation of
Remark \ref{rem:thm} that $a \in (0,1)$ and $b=1$.

The papers  \cite{Veselic-02a,KostrykinV-06,Veselic-circulant} 
establish Wegner estimates for both alloy-type Schr\"odinger operators on $L^2(\RR^d)$
and discrete alloy-type Schr\"odinger operators on $\ell^2(\ZZ^d)$.
Since the present paper concerns the latter model we will discuss here first the 
results of  \cite{Veselic-02a,KostrykinV-06,Veselic-circulant}  refering to operators on the lattice. 
For the discrete alloy-type model on  $\ell^2(\ZZ^d)$, \cite{Veselic-circulant}
establishes a Wegner estimate analogous to Corollary~\ref{cor:dos} above,
under the additional assumption that the function
\begin{equation}
\label{eq:symbol}
s\colon \theta\mapsto s(\theta) :=\sum_{k \in \ZZ^d} u(k) \mathrm e^{-\mathrm i k \cdot \theta}
\text{ does not vanish on $[0, 2\pi  )^d $. }
\end{equation}
To be able to compare the two results, note that $\bar u := \sum_{k\in\ZZ^d} u(k)=s(0)$. Thus assumption 
\eqref{eq:symbol} requires that the image of the set $[0, 2\pi  )^d $ under $s$ does not meet $0\in\CC$
whereas the assumption in  Corollary \ref{cor:dos} requires this property for the image of the set $\{0\}$ only. 
The later condition is generically satisfied. Let us now turn to the situation when $\bar u=0$.
Special cases of this class of single site potentials
are covered by Theorem 2 in \cite{KostrykinV-06} and Exp. 10 in \cite{Veselic-02a}.
They correspond to special cases of Theorem \ref{thm:degenerate} and do not give an as explicit control 
over the volume dependence of the Wegner bound.

Let us say a few words which ideas are used in the proofs to overcome the restrictions 
imposed on the single site potentials in \cite{KostrykinV-06,Veselic-circulant}. There a transformation of the random variables is used to construct a 
non-negative linear combination of translates of single site potentials. The price to pay is that the new transformed random variables are no longer
independent. The argument of \cite{KostrykinV-06,Veselic-circulant}  uses the inverse transformation on the probability space to recover 
in a later step of the proof independence again. This leads to an uniform invertibility requirement for a sequence of a certain auxiliary Toeplitz 
or circulant matrices constructed from the values of the single site potential $u$. 
Condition \eqref{eq:symbol}  on the function $s$ ensures that this invertibility property holds.
The proof of the present paper uses a similar transformation of the coordinates of the product probability space, 
but the inverse transformation is no longer needed.  This leads to  less stringent conditions on the single site potential $u$.

Contrary to the present paper \cite{Veselic-02a,KostrykinV-06,Veselic-circulant}  give Wegner estimates for 
continuum alloy-type Schr\"odinger operators on $L^2(\RR^d)$ as well.
The bounds are linear in the volume of the box
and Lipschitz continuous in the energy variable.
The bound is valid for all compact intervals along the energy axis.
These bounds are valid for single site potentials
$u \in L_c^\infty(\RR^d)$ which have a generalised step function form
and satisfy a condition analogous to
\eqref{eq:symbol}.

Let us stress that Wegner estimates for sign changing single site potentials
are harder to prove for operators on $L^2(\RR^d)$ than for ones on $\ell^2(\ZZ^d)$.
The reason is that for discrete models we have in the randomness
a degree of freedom for each point in the configuration space $\ZZ^d$.
For the continuum alloy-type model the configuration space is $\RR^d$
while the degrees of freedom are indexed by a much smaller set, namely
$\ZZ^d$.

The role played by a Wegner estimate in the framework of a localisation proof using the multiscale analysis 
is analogous to role played by the finiteness of the expectation of fractional powers of the Green's function
for fractinal moment method.
Recently a fractional moment bound for the alloy-type model on $\ell^2(\ZZ)$ has been proven in \cite{ElgartTV}.
(See also \cite{TautenhahnV-10} for a related result.)
It holds for arbitrary compactly supported single site potentials.
The result can be extended to the one-dimensional strip, while the extension to $\ZZ^d$ is unclear at the moment.

Another important class  of  random Hamiltonians exhibiting non-monotone dependence 
on the random variables  are Schr\"odinger operators with random magnetic fields. 
Wegner estimates for such models are established
 \cite{HislopK-02,KloppNNN-03,Ueki-08}. In particular, \cite{KloppNNN-03} gives a Wegners for a
random magnetic field Hamiltonian on the lattice $\ell^2(\ZZ^2)$ and is thus comparable with results in the present paper.
It is not clear whether our methods can be used to treat the model of \cite{KloppNNN-03} 
since is is necessary to find a set of transformed random variables which produces a perturbation of fixed sign.
(For discrete alloy-type models studied here this is done in Sections \ref{s:transformation-degenerate}  and \ref{s:transformation-degenerate-exponential}.)
Since the structure of the randomness is different in disordered magnetic field models, 
it is not clear  whether such an transformation exists.

Let us also mention the random displacement model as an important example of random Schr\"odinger operators exhibiting non-monotone parameter dependence.
For such models in the continuum the location of the minimum of the spectrum, Lifschitz tails and Wegner estimates have been studied in 
 \cite{LottS-02,BakerLS-08,BakerLS-09, KloppN-09a,KloppN-09b, GhribiK-10}. These models do not have a direct analog on the space 
$\ell^2(\ZZ^d)$ due to the lack of continuous deformations.

Very recently Kr\"uger \cite{Krueger} has obtained results on localisation for a class of discrete alloy type models which includes the 
ones considered here.
The results rely  on the multiscale analysis and the use of Cartan's 
lemma in the spirit as is has been used earlier, e.g. in  \cite{Bourgain-09}.

\section{An abstract Wegner estimate and the proof of Theorem \ref{thm:non-degenerate}}

An important step in the proofs of the Theorems of the last section is an
abstract Wegner estimate which we formulate now.
We abbreviate in the sequel the characteristic function $\chi_{\Lambda_L}$
by $\chi_{\Lambda}$.

\begin{lem} \label{thm:abstract}
Let $L\in \NN, E \in \RR, \epsilon>0$ and $I:=[E-\epsilon,E+\epsilon]$.
Denote by  $E(\omega, L,n)$ the $n$-th eigenvalue of the operator $H_{\omega,L}$.
Assume that there exist an $\delta >0$ and $a_L\in \ell^1(\ZZ^d)$
such that for all $n$
\begin{equation}
 \label{eq:positive}
 \sum_{k \in \ZZ^d} a_L(k) \frac{\partial }{\partial \omega_k } E(\omega,L,n)  \ge \delta
\end{equation}
Then
\begin{equation*}
\EE(\Tr \chi_I(H_{\omega,L})  )
\le
\frac{4\epsilon}{\delta}
 \sum_{k \in \ZZ^d} |a_L(k)|\, \|f\|_{BV} \, \rank (\chi_\Lambda u(\cdot -k)  )
\end{equation*}
\end{lem}
\bigskip
Since $a_L \in \ell^1$ and the derivatives $\frac{\partial }{\partial \omega_k } E(\omega,L,n) $ are uniformly bounded,
the sum \eqref{eq:positive} is absolutely convergent.
Note that one can always replace the sum $\sum_{k \in \ZZ^d} $ by
$\sum_{k \in \Lambda_L^+}$.
Here $\Lambda_L^+=\{k \in \ZZ^d\mid u(\cdot-k)\cap \Lambda_L \neq \emptyset\}$
denotes the set of lattice points such that the corresponding coupling constant
influences the potential in the box $\Lambda_L$.
In particular, if the support of $u$ is contained in $[-n,\dots,0]^d$,
the sum reduces to $\sum_{k \in \Lambda_{L+n}}$.

Note that the sequence $a_L$ may be chosen differently for different cubes $\Lambda_L$.
In our applications, namely the proofs of Theorems \ref{thm:degenerate}, \ref{thm:non-degenerate}, and \ref{thm:degenerate-exponential},
we will find a fixed sequence $a$, not necessarily in $\ell^1(\ZZ^d)$, such that appropriate finite truncations
 give the desired coefficients $a_L(k)$ adapted for a cube $\Lambda_L$ of size $L$. Note that for $u$ with compact support, the function
$k \mapsto \rank (\chi_\Lambda u(\cdot -k)  )$ already implements the truncation:
the terms with  $k$ outside $\Lambda_L^+$ do not contribute to the sum. In this situation the condition $a_L\in \ell^1$ is not needed.

We give a simple sufficient condition which ensures the hypothesis of Lemma~\ref{thm:abstract}.
\begin{cor}\label{cor:abstract}
Let $L\in \NN, ¸\epsilon>0$ and $I:=[E-\epsilon,E+\epsilon]$.
Assume that there exist an $\delta >0$ and $a_L\in \ell^1(\ZZ^d)$
such that  all $x \in \Lambda_L$
\[
 \sum_{k \in \ZZ^d} a_L(k) u(x- k) \ge \delta
\]
Then
\begin{equation*}
\EE(\Tr \chi_I(H_{\omega,L})  )
\le
\frac{4\epsilon}{\delta}
 \sum_{k \in \ZZ^d} |a_L(k)|\, \|f\|_{BV} \, \rank (\chi_\Lambda u(\cdot -k)  )
\end{equation*}
\end{cor}
\begin{proof}
By first order perturbation theory, respectively the Hellmann-Feynman formula we have
\[
\frac{\partial }{\partial \omega_k } E(\omega,L,n)
=
\langle \psi_n, u(\cdot-k)\psi_n\rangle
\]
where $\psi_n$ is the normalised eigensolution to $H_{\omega,L} \psi_n =E(\omega,L,n) \psi_n$.
Thus
\[
\sum_{k \in \ZZ^d} a_L(k) \frac{\partial }{\partial \omega_k } E(\omega,L,n)
=\sum_{k \in \ZZ^d} a_L(k) \langle \psi_n, u(\cdot-k)\psi_n\rangle \ge \delta
\]
\end{proof}

The proof of Lemma \ref{thm:abstract}  relies on quite standard techniques, see e.g.~\cite{Wegner-81,Kirsch-96,HislopK-02,LenzPPV-08}.
The main point of the Lemma is that it singles out
a relation between properties of linear combinations of single site potentials and a Wegner estimate.
In the course of the proof we will need the following estimate, which is related to the spectral shift function.
Recall that $ n \mapsto E(\omega, L,n)$ is an enumeration of the eigenvalues of $H_{\omega,L}$.
\begin{lem} \label{lem:partial-integration}
Let $ f\colon \RR\to \RR$ be a function in $BV \cap L^1(\RR)$, $\rho \in C^\infty(\RR)$,
$k \in \ZZ^d$ and $s \in \RR$.
Then
\begin{equation*}
\sum_{n\in \NN} \int d\omega_k  f (\omega_k )  \frac{\partial }{\partial \omega_k } \rho(E(\omega,L,n)  +s)
\le \|f\|_{BV} \, \rank (\chi_\Lambda u(\cdot -k)  )\, \int |\rho'(x)|dx
\end{equation*}
\end{lem}
Note that if $ k \not \in \Lambda_{L}^+$ then
$\frac{\partial }{\partial \omega_k } E(\omega,L,n)=0$.
Also note that the sum over $n$ is in fact finite since $H_{\omega,L}$
is defined on a finite dimensional vector space.

\begin{proof}
We will use that if $ g \in C^\infty$ and  $f \in BV \cap L^1$ the partial integration bound
\[
\int f(x) g'(x) dx \le \|g\|_\infty \|f\|_{BV}
\]
holds.
Denote by  $E(\omega, \omega_k=0,L,n)$ the $n$-th eigenvalue of the operator
$H_{\omega,\omega_k=0,L}:=H_{\omega,L} -\omega_k u(\cdot-k) $ on $\ell^2(\Lambda_L)$. Partial integration yields
\begin{align*}
&\sum_{n\in \NN} \int d\omega_k  f (\omega_k )  \frac{\partial }{\partial \omega_k } \rho(E(\omega,L,n)  +s)
 \\
&=
\int d\omega_k  f (\omega_k )  \frac{\partial }{\partial \omega_k } \sum_{n\in \NN}  \Big(\rho(E(\omega,L,n)  +s)    -\rho(E(\omega, \omega_k=0,L,n)  +s)     \Big)
 \\
&\le \|f\|_{BV}
\sup_{\omega_k \in \supp f}\Big|\sum_{n\in \NN} (\rho(E(\omega,L,n)  +s)    -\rho(E(\omega, \omega_k=0,L,n)  +s) )    \Big|
\end{align*}
Here we used that $\omega_k\mapsto E(\omega,L,n)$ is an infinitely differentiable function
cf. \cite{Kato-66}.
Now
\begin{align*}
\sum_{n\in \NN} \rho(E(\omega,L,n)  +s)    -\rho(E(\omega, \omega_k=0,L,n)  +s)
\\
=\Tr \Big( \rho((H_{\omega,L}  +s)    -\rho((H_{\omega,\omega_k=0,L}   +s) \Big)
\end{align*}
can be expressed in terms of the spectral shift function
$\xi(\cdot, H_{\omega,L}, H_{\omega,\omega_k=0,L})$ of the operator pair
$H_{\omega,L}, H_{\omega,\omega_k=0,L}$ as
\begin{align*}
\int \rho'(x) \xi(x, H_{\omega,L}   , H_{\omega,\omega_k=0,L}) dx.
\end{align*}
Since $\|\xi\|_\infty $ is bounded by the rank of the perturbation $\chi_\Lambda u(\cdot -k) $, we obtain
\begin{align*}
\sum_{n\in \NN} \rho(E(\omega,L,n)  +s)    -\rho(E(\omega, \omega_k=0,L,n)  +s)
\le \rank (\chi_\Lambda u(\cdot -k)  ) \, \int |\rho'|
\end{align*}
and the proof of the Lemma is completed.
\end{proof}

Now we turn to the proof of Lemma \ref{thm:abstract}.

\begin{proof}[Proof of Lemma \ref{thm:abstract}]
Let  $\rho\in C^\infty(\RR)$ be a non-decreasing function such that on $(-\infty, -\epsilon]$ it is identically equal to $-1$,
on $[\epsilon, \infty)$ it is identically equal to zero and $\|\rho'\|_\infty \le 1/\epsilon$.
By the chain rule we have
\begin{multline*}
 \sum_{k \in \ZZ^d} a_L(k)\frac{\partial }{\partial \omega_k } \rho(E(\omega,L,n)  -E+t)
\\
=
\rho'(E(\omega,L,n)  -E+t)  \sum_{k \in \ZZ^d} a_L(k)\frac{\partial }{\partial \omega_k } E(\omega,L,n)
\end{multline*}
The assumption  \eqref{eq:positive} implies now
\begin{equation*}
\rho'(E(\omega,L,n)  -E+t)
\le
\frac{1}{\delta}
 \sum_{k \in \ZZ^d} a_L(k)\frac{\partial }{\partial \omega_k } \rho(E(\omega,L,n)  -E+t)
\end{equation*}
Since $\chi_I \le \int_{-2\epsilon}^{2\epsilon} dt \, \rho'(x-E+t)$
for $I :=[E-\epsilon, E + \epsilon]$
we have
\begin{equation*}
\Tr \chi_I(H_{\omega,L})
\le
\frac{1}{\delta}
\int_{-2\epsilon}^{2\epsilon} dt \sum_{n\in \NN}
 \sum_{k \in \ZZ^d} a_L(k)\frac{\partial }{\partial \omega_k } \rho(E(\omega,L,n)  -E+t)
\end{equation*}
Note that for a random variable $F\colon\Omega \to \RR$ we have
$\EE(F) = \EE (\int f(\omega_k) d \omega_k F(\omega))$
Thus using Lemma \ref{lem:partial-integration}  and
$\int |\rho'(x) |dx=1$  we obtain
\begin{equation*}
\EE(\Tr \chi_I(H_{\omega,L})  )
\le
\frac{4\epsilon}{\delta}
 \sum_{k \in \ZZ^d} |a_L(k)|\, \|f\|_{BV} \ \rank (u\cdot \chi_\Lambda)
\end{equation*}
\end{proof}

Now we are in the position to give a
\begin{proof}[Proof of Theorem \ref{thm:non-degenerate}]
Let  $\psi_n$ be a normalised eigenfunction associated to $E(\omega,L,n)$ and
$Q(L,m)= \bigcup_{k \in \Lambda_L} \Big( k + [-m,m]^d \cap \ZZ^d\Big)$.
W.l.o.g.{} we may assume $\bar u >0$.
Then $\sum_{k \in Q(L,m)} u(k)\ge \bar u/2$.
Choose now the coefficients in Corollary \ref{cor:abstract}  in the following way:
$a_L(k) = 1$ for $ k \in Q(L,m)$ and  $a_L(k) = 0$ for $ k $ in the complement of $Q(L,m)$.
Then
\[
 \sum_{k \in \ZZ^d} a_L(k) \langle \psi_n, u(\cdot-k)\psi_n\rangle
= \langle \psi_n, \sum_{k \in Q(L,m)}  u(\cdot-k)\psi_n\rangle
\ge \bar u / 2.
\]
\end{proof}

\begin{proof}[Proof of Corollary \ref{cor:dos}]
Set $a_L(k)= 1$ for $ k \in \Lambda_{L+n}$ and  $a_L(k) = 0$ for $ k $ in the complement of $\Lambda_{L+n}$. Then
\[
 \sum_{k \in \ZZ^d} a_L(k) \langle \psi_n, u(\cdot-k)\psi_n\rangle
= \langle \psi_n, \sum_{k \in \Lambda_{L+n}}  u(\cdot-k)\psi_n\rangle
= \bar u
\]
An application of Corollary \ref{cor:abstract}  now completes the proof.
\end{proof}

\section{Proof of Theorem \ref{thm:degenerate}}
\label{s:transformation-degenerate}

In this section we give a proof of Theorem \ref{thm:degenerate}.
In view of Theorem \ref{thm:non-degenerate} it is sufficient to consider the case that
the single site potential $ u \colon \ZZ^d \to \RR, u \in \ell^1(\ZZ^d)$ is
degenerate in the sense that $\sum_{x\in \ZZ^d} u(x)=0$.
We explain how to find in this situation an appropriate linear combination of single site potentials --- or, equivalently,
an appropriate linear transformation of the random variables --- which can be efficiently used for averaging. The aim of the
linear transformation is to extract a perturbation potential which is strictly positive on the box $\Lambda$.

Let us first consider the case $d=1$.  Then we can assume without loss of generality that $\supp u \subset \{-n, \dots 0\}$. For a given cube $\Lambda_L = \{ 0, \dots, L\}$
we are looking for an array of numbers $a_k, k \in \Lambda_{L+n}$ such that
 we have
\begin{equation}
\label{e:linear-combination}
 \sum_{k \in \Lambda_{L+n}}  a_k u(x-k) = constant  \, >0 \quad \text{ for all $x \in \Lambda_L$}
\end{equation}
In fact, we will find a sequence of numbers $a_k, k \in \NN$ such that  we have
\begin{equation}
\label{e:recursion}
 \sum_{k \in \NN}  a_k u(x-k) = constant  \, >0 \quad \text{ for all $x \in \NN$}
\end{equation}
If we truncate this sequence, we obtain an array of numbers satisfying \eqref{e:linear-combination}.

For a function $F\colon (1-\epsilon, 1+\epsilon) \to \RR$  with $ \epsilon>0$
we say that it has a root of order $m\in \{ 0,\dots,n\}$ at $t=1$ iff it is in $C^m(1-\epsilon, 1+\epsilon)$ and
\begin{align}
\label{e:homo-solution}
\Big( \frac{d^j}{dt^j} F(t) \Big)\Big|_{t=1}&= 0 \quad \text{ for } j =0, \dots,m-1 \\
\label{e:inhomo-solution}
c(F):= \Big( \frac{d^m}{dt^m} F(t) \Big)\Big|_{t=1}&\neq 0
\end{align}
In particular, $m=0$ means that $F(1)\neq0$. If $F$ is a polynomial of degree not exceeding $m$,
if  \eqref{e:homo-solution}
holds and in addition $c(F)=0$, then $F\equiv 0$.
In this case we say that $F$ has a root of infinite order at $t=1$.

Given a function $w \colon \ZZ \to \RR$ such that $F_w(t) := \sum_{\nu\in \ZZ} t^\nu w(-\nu)$
converges for $t \in (1-\epsilon, 1+\epsilon)$
we call $(1-\epsilon, 1+\epsilon)\ni t \mapsto F(t) := F_w(t)$ the accompanying (Laurent) series of $w$.
If $\supp w \subset \{-n,\dots,0\}$
we call $t \mapsto p(t) := p_w(t) :=\sum_{\nu=0}^n t^\nu w(-\nu)$ the accompanying polynomial of $w$.

%

\begin{lem} \label{l:infinite-recursion-solution}
Let $ D \in \NN_0$ and $ a_k = k^D$ for all $ k \in \NN$. Let $ m$ be the order of the
root $t=1$ of the Laurent series $F$ accompanying the function $w\colon \ZZ \to \RR$
with convergent series $\sum_{\nu\in \ZZ} t^\nu w(-\nu)$ for
$t \in  (1-\epsilon, 1+\epsilon)$.
\begin{enumerate}[(a)]
 \item If $ m>D$ then $ \sum_{k \in \ZZ} a_k w(x-k)=0$ for all $x \in \NN$.
 \item If $ m= D$ then $ \sum_{k\in \ZZ} a_k w(x-k)=c(F)$ for all $x \in \NN$.
\end{enumerate}
\end{lem}

An important and well known special case is

\begin{cor}
\label{l:recursion-solution}
Let $ D \in \NN_0$ and $ a_k = k^D$ for all $ k \in \NN$. Let $ m$ be the order of the
root $t=1$ of the polynomial $p$ accompanying the function $w\colon \ZZ \to \RR$ with
$\supp w \subset \{-n, \dots, 0 \}$.
\begin{enumerate}[(a)]
 \item If $ m>D$ then $ \sum_{k =x}^{x+n} a_k w(x-k)=0$ for all $x \in \NN$.
 \item If $ m= D$ then $ \sum_{k =x}^{x+n} a_k w(x-k)=c(p)$ for all $x \in \NN$.
\end{enumerate}
\end{cor}
\noindent
Due to the support condition
$ \sum_{k \in \NN}a_k w(x-k)=\sum_{k =x}^{x+n} a_k w(x-k)$ for all $x \in \NN$.

\begin{proof}[Proof of Lemma \ref{l:infinite-recursion-solution}]
 First note that for arbitrary $\nu \in \NN$ and $ s\in \RR $ we have
\[
 \frac{d^\nu}{ds^\nu} F(e^s) = \sum_{\kappa =1}^\nu c_\kappa \, F^{(\kappa)} (e^s) \, e^{\kappa s}
\]
with some $c_1, \dots, c_{\nu-1} \in \NN_0$ and $c_\nu=1$.
For the value $s=0$ it follows from \eqref{e:homo-solution}
that $\frac{d^\nu}{ds^\nu} F(e^s) = 0$ for $\nu =0, \dots,m-1$ and
from \eqref{e:inhomo-solution} that
$\frac{d^m}{ds^m} F(e^s)= F^{(m)}(e^s) \, e^{m s}= c(F)$.

We note that $a_k = \frac{d^D}{ds^D} e^{k s}$ for $s=0$ and insert this into the LHS of \eqref{e:recursion} to obtain
\begin{align}\nonumber
\sum_{k \in \ZZ} a_k w(x-k)
&=
\sum_{k \in \ZZ} w(x-k) \frac{d^D}{ds^D} e^{k s}
=
\sum_{\nu \in \ZZ} w(-\nu) \frac{d^D}{ds^D} e^{(\nu+x)s}
\\
\label{e:infinite-insert}
&=
\frac{d^D}{ds^D} \big(e^{xs} \, F(e^s)\big)
=
\sum_{r=0}^D  \genfrac{(}{)}{0pt}{}{D}{r} \Big( \frac{d^r}{ds^r} F(e^s) \Big)
\Big( \frac{d^{D-r}}{ds^{D-r}} e^{xs} \Big) \, .
\end{align}
For $s=0$, \eqref{e:infinite-insert} vanishes if $D<m$ and equals $c(F)$ if $D=m$.
\end{proof}

Thus we have found in the case $d=1$ and $w=u$ a linear combination with the desired property \eqref{e:recursion}.
In the multidimensional situation we will reduce the dimension one by one
and construct from a non-vanishing single site potential in dimension $j$
a non-vanishing one in dimension $j-1$. In each reduction step we
apply Corollary \ref{l:recursion-solution}.

Let $\wj\colon \ZZ^j \to \RR$ be compactly supported and not identically vanishing.
W.l.o.g. we assume $\supp \wj \subset [-n,0]^j\cap \ZZ^j$.
Next we define a `projected' single site potential  $\wjj \colon \ZZ^{j-1} \to \RR$ as follows.
Consider the family of polynomials $p(x_1, \dots, x_{j-1}, \cdot) \colon \RR \to \RR$,
indexed by $(x_1, \dots, x_{j-1}) \in \{-n,\dots, 0\}^{j-1} $ and defined by
\begin{equation}
\label{e-polynomial}
p(x_1, \dots, x_{j-1}, t) := \sum_{\nu=0}^n t^\nu \, \wj(x_1, \dots, x_{j-1}, -\nu) \, .
\end{equation}
Let $m(x_1, \dots, x_{j-1})\in \{ 0,\dots,n, \infty\}$ be the order of the root $t=1$ of the polynomial $p(x_1, \dots, x_{j-1}, \cdot)$ and $M:=M_{j} := \min \big\{ m(x_1, \dots, x_{j-1}) \mid x_1, \dots, x_{j-1} \in \{-n, \dots, 0\}\big\}$ the minimal degree occurring in the family.
Since $\wj$ does not vanish identically, $M_{j} \le n$.
Set 
\begin{align*}
I_{j-1}& := \{ (x_1, \dots, x_{j-1})\in \{-n,\dots, 0\}^{j-1} \mid m(x_1, \dots, x_{j-1})=M_{j}\}  \ \text{ and }
\\
J_{j-1} &:= \{ (x_1, \dots, x_{j-1})\in \{-n,\dots, 0\}^{j-1} \mid m(x_1, \dots, x_{j-1})>M_{j}\}  
\end{align*}
\begin{lem}
For all $(x_1, \dots, x_{j-1})\in \{-n,\dots, 0\}^{j-1}$
we have the equality
\begin{equation}
\label{e:def-wjj}
\sum_{k\in \NN}k^{M} \, \wj(x_1, \dots, x_{j-1},x_j-k) =
\Big (\frac{d^M}{dt^M} p(x_1, \dots, x_{j-1},t) \Big) \Big|_{t=1} \, .
\end{equation}
We denote the function in \eqref{e:def-wjj} by $ \wjj\colon \ZZ^j \to \RR$. Then $ \wjj$ is
independent of the variable $x_j$ and therefore we call it the single site potential in reduced dimension and consider it sometimes as a function $ \wjj\colon \ZZ^{j-1} \to \RR$. Its support is contained in $\{-n,\dots, 0\}^{j-1}$.

Moreover, $\wjj(x_1, \dots, x_{j-1})=0$ if $(x_1, \dots, x_{j-1}) \in J_{j-1}$ and
$\wjj(x_1, \dots, x_{j-1})\neq0$ if $(x_1, \dots, x_{j-1}) \in I_{j-1}$.
\end{lem}

\begin{rem} The lemma establishes in particular that
\begin{itemize}
\item $M$ is an element of $\{0, \dots,n\}$. If we had $M \ge n+1$, then all polynomials $p(x_1, \dots, x_{j-1}, \cdot)$ would vanish identically and thus $\wj \equiv 0$ contrary to our assumption.
\item $\wjj$  does not vanish identically. In fact $\supp \wjj = I_{j-1} \neq \emptyset$ by definition.
 \end{itemize}
\end{rem}

\begin{proof}
Consider first the case $(x_1, \dots, x_{j-1}) \in J_{j-1}$. Then for any $x_j \in \NN$
\[
\wjj (x_1, \dots, x_{j-1}) =\sum_{k \in \NN}k^{M} \, \wj(x_1, \dots, x_{j-1},x_j-k) =0
\]
by Lemma \ref{l:recursion-solution}, part (a), since $t=1$ is a root of order
$M+1$ or higher of the accompanying polynomial $p(x_1, \dots, x_{j-1}, \cdot)$.

Now, if $ (x_1, \dots, x_{j-1}) \in I_{j-1}$  then the order of the root $t=1$ of the polynomial
$p(x_1, \dots, x_{j-1}, \cdot)$ equals $M$. Thus by part (b) of Lemma
\ref{l:recursion-solution}
\begin{align*}
\wjj (x_1, \dots, x_{j-1})
&=\sum_{k \in \NN}k^{M} \, \wj(x_1, \dots, x_{j-1},x_j-k) \\
&=\Big (\frac{d^M}{dt^M} p(x_1, \dots, x_{j-1},t) \Big) \Big|_{t=1}
\end{align*}
for all $ x_j \in \NN$.
\end{proof}

In the last step $j=1 \to j-1 =0$ of the induction we obtain a reduced single site potential
\[
w^{(0)} = \Big (\frac{d^{M_1}}{dt^{M_1}} p(t) \Big) \Big|_{t=1} = c(p)
\]
which is a simply non-zero real.

Now we describe the result which is obtained after the reduction is applied $d$ times.
Given a  single site potential $u\colon\ZZ^d \to \RR$
with $ \supp u \subset [-n,0]^d \cap \ZZ^d$, set $w^{(d)}=u$ and
\begin{align}
w^{(0)}
& = \sum_{k_1 \in \NN } k_1^{M_1} w^{(1)} (x_1-k_1) \\
& = \sum_{k_1 \in \NN } k_1^{M_1} \dots
\sum_{k_d \in \NN } k_d^{M_{d}} w^{(d)} (x_1-k_1, \dots , x_d-k_d)
\end{align}
Thus we have produced a linear combination of single site potentials
\[
\sum_{k \in \Lambda_{L+n} } b_k w^{(d)} (x_1-k_1, \dots , x_d-k_d)
\quad \text{ where }  \quad b_k :=k_1^{M_1} \dots k_d^{M_{d}}
\]
which is a constant, non-vanishing function on the cube $\Lambda_L $.
Moreover, the coefficients satisfy the bound
\[	
|b_k| \le k_1^n \dots k_d^n \le (L+n)^{d\cdot n}
\quad \text{ for all } \quad  k \in \Lambda_{L+n}
\]
Now an application of Corollary \ref{cor:abstract} with the choice
$a_L(k)=b_k$ for $k \in \Lambda_{L+n}$ and
$a_L(k)=0$ for $k$ in the complement of this set
completes the \emph{Proof of Theorem \ref{thm:degenerate}}.

\section{Proof of Theorem \ref{thm:degenerate-exponential}}
\label{s:transformation-degenerate-exponential}

The assumption on the exponential decay of $u$  implies that
$F(z) =\sum_{\nu \in \ZZ} z^\nu u(-\nu)$ is an
absolutely and uniformly convergent Laurent series on the  annulus
$\{z \in \CC \mid r_1 \le |z| \le r_2\}$ for some $0< r_1 < 1 < r_2 < \infty $ and represents
there a holomorphic function. This implies that there exists a $D \in\NN_0$
such that $ c(F) := \frac{\partial^D}{\partial z^D} F(z) \mid _{z=1} \neq 0$.
Otherwise $F$ would be identically vanishing, implying that $u$ vanishes identically.
Thus the root $z=1$  of $F$ has a well defined, finite order $D\in \NN_0$
and Lemma \ref{l:infinite-recursion-solution} can be applied.

The problem is now that the series $\sum_{k \in \ZZ} k^D$ is not absolutely convergent.
For this reason we will replace it with an appropriate finite cut-off sum.
Assume in the following w.l.o.g.{} that $c(F) >0$.
A lengthy but easy calculation shows that for all $ \beta > D / | \log s|$ there exists a constant $K_\beta \in (0,\infty)$ such that for all $L\in \NN$
\[
\forall \  x \in \Lambda_L :
\sum_{k \not \in \{-K/2, \dots, m\}}
|k|^D     |u(x-k)| \le \frac{c(F)}{2}
\]
where $ m= L + \beta \log L + K_\beta/2$. Consequently
\[
\forall \  x \in \Lambda_L :
\sum_{k \in \{-K/2, \dots, m\}}
k^D     u(x-k) \ge \frac{c(F)}{2}
\]
Thus we can apply Corollary \ref{cor:abstract}
with the choice
$a_L(k) = k^D$ for $ k \in \{-k, \dots, m\}$
and $a_L(k) = 0$ for $ k \in \{-k_1, \dots, m+1\}$
and obtain
\[
\EE \left \{\Tr \big [\chi_{[E-\epsilon,E+\epsilon]}(H_{\omega,L})\big]\right\}
\le
\frac{8 \epsilon}{c(F)}  \|f\|_{BV} \ L (L + \beta \log L + K)^{D+1}
\]

%
%
%

\section{Discussion: Localisation for large disorder}

In the case that $H_0= \Delta$ is the finite difference Laplacian and the disorder is sufficiently strong, Theorem 
\ref{thm:degenerate} can be used to prove 
exponential localisation on the whole energy axis $\RR$, i.e.~to show that almost surely $H_\omega, \omega \in \Omega$
has no continuous spectral component and that all eigenfunctions of $H_\omega$ decay exponentially at infinity.
We do not discusslocalisation near spectral edges which is a more delicate issue.
Note that Theorem \ref{thm:degenerate} assumes in particular that the single
site potential $u$ is of compact support. 

The results described below are based on the multiscale analysis, cf.~e.g.~
\cite{FroehlichS-83,DreifusK-89,GerminetK-01a}.
It  is an induction procedure over increasing length scales $L_k, k \in \NN$.
The induction step uses a Wegner estimate to deduce from probabilistic decay estimates 
on the Green's function of the random operator restricted to a box of size $L_k$
corresponding decay estimates on the larger scale $L_{k+1}$.
The induction anchor is provided by the initial scale estimate,		
a probabilistic statement on the decay of the Green's function of the
random operator restricted to a box on first scale $L_0$. 
A very strong form of the initial scale estimate is that for some $p>d , m>0$ 
\begin{equation}
 \label{eq:uniform-initial-scale-estimate}
 \PP\{ \|(H_{\omega,L}-E)^{-1}\| \le \exp (-mL/2)\} \ge 1- L_0^p. 
\end{equation}	
Together with a Wegner estimate as in Theorem \ref{thm:degenerate}, the bound \eqref{eq:uniform-initial-scale-estimate}
yields exponential localisation for $H_\omega, \omega\in  \Omega$ in a small neighbourhood of $E$, provided that $L_0$ 
is larger than a certain  critical lenght scale $L^*$, dependig on the parameters of the model. 

Here we present a simple idea how to derive the initial scale estimate from the 
Wegner estimate in the case of large disorder, which we learned from A.~Klein and which 
has almost the same proof as Theorem 11.1 in \cite{Kirsch-08}
although the statements and the models under consideration are somewhat different.
For an earlier related result see Proposition A.1.2 in \cite{DreifusK-89}.

\begin{lem}
\label{l:large-disorder}
Let the assumptions of Theorem \ref{thm:degenerate} hold. Let $H_0=\Delta$ and $p \in \NN$. 
Choose $L \in \NN$ such that $e^L \ge (c_u \rank u ) L^p (L+n)^{d(n+1)}$. 
If $\|f\|_{BV}^{-1/2}\ge e^L $, then 
\begin{equation*}
\forall \ E \in \RR \colon \ \PP \left \{\| (H_{\omega,L} -E)^{-1} \|  >e^{-L}\right\}       \le L^{-p}
\end{equation*}
\end{lem}
Here the quantity  $\|f\|_{BV}^{-1}$ is a measure for the disorder:
if it is large the values of the corresponding random variable are spread out over a large interval.
Thus the assumption  $\|f\|_{BV}\le e^{-2L}$  describes a large disorder regime.
\begin{proof}
Since $ \| (H_{\omega,L} -E)^{-1}\| = d(\sigma(H_{\omega,L}), E)^{-1}$,  we have
\begin{multline*}
\PP \{\| (H_{\omega,L} -E)^{-1}\|  >e^{-L}\} 
=
\PP \{d(\sigma(H_{\omega,L}), E) < e^{L}\} 
\\
=
\PP \left \{(E-e^L,E+e^L) \cap \sigma(H_{\omega,L})\neq \emptyset\right\}
\le 
\EE \left \{\Tr \big [\chi_{[E-e^{L},E+e^{L}]}(H_{\omega,L})\big]\right\}
\\
\le 
c_u \, \rank u   \|f\|_{BV} \, e^L\, (L+n)^{d(n+1)} 
\le  e^{-L}
c_u \, \rank u   \, (L+n)^{d(n+1)} 
\end{multline*}
which is bounded by $L^{-p} $ by our assummtion.
\end{proof}
This lemma establishes  an initial scale decay estimate \eqref{eq:uniform-initial-scale-estimate}
for a small neighbourhood of an arbitrary energy. 
However increasing the disorder means changing the model and in particular increasing 
the sup-norm of the single site potential.
Thus one has to check on which parameters of the single site the critical scale $L^*$ depends. 
Indeed, $L^*$ does depend on the size of the support, 
but not on the supremum norm of the single site potential. This fact can be seen from the original proof of \cite{DreifusK-89}. 
A detailed analysis how $L^*$ depende on various model parameters 
has been worked out for continuum random operators in \cite{GerminetK-01a} and applies to discrete operators anagolously.

\subsection*{Acknowledgements}
The author would like to thank A.~Klein for pointing out the reasoning behind Lemma 
\ref{l:large-disorder} and anonymous referees for helpful comments. 
It was a pleasure to have stimulating discussions 
concerning discrete alloy type models 
with A.~Elgart, H.~Kr\"uger, G.~Stolz, and M.~Tautenhahn.


\begin{thebibliography}{10}

\bibitem{BakerLS-08}
J.~Baker, M.~Loss, and G.~Stolz.
\newblock Minimizing the ground state energy of an electron in a randomly
  deformed lattice.
\newblock {\em Comm. Math. Phys.}, 283(2):397--415, 2008.

\bibitem{BakerLS-09}
J.~Baker, M.~Loss, and G.~Stolz.
\newblock Low energy properties of the random displacement model.
\newblock {\em J. Funct. Anal.}, 256(8):2725--2740, 2009.

\bibitem{Bourgain-09}
J.~Bourgain.
\newblock An approach to {W}egner's estimate using subharmonicity.
\newblock {\em J. Stat. Phys.}, 134(5-6):969--978, 2009.


\bibitem{BourgainK-05}
J.~Bourgain and C.~E. Kenig.
\newblock On localization in the continuous {A}nderson-{B}ernoulli model in
  higher dimension.
\newblock {\em Invent. Math.}, 161(2):389--426, 2005.

\bibitem{DreifusK-89}
H.~v. Dreifus and A.~Klein.
\newblock A new proof of localization in the {A}nderson tight binding model.
\newblock {\em Comm. Math. Phys.}, 124(2):285--299, 1989.

\bibitem{ElgartTV}
A.~Elgart, M.~Tautenhahn, and I.~Veseli\'c.
\newblock Localization via fractional moments for models on $\mathbb{Z}$ with
  single-site potentials of finite support.
\newblock preprint http://arxiv.org/abs/0903.0492.

\bibitem{GhribiK-10}
G.~F. and F.~Klopp.
\newblock Localization for the random displacement model at weak disorder.
\newblock {\em Ann. Henri Poincar\'e}, 2010.

\bibitem{FroehlichS-83}
J.~Fr\"ohlich and T.~Spencer.
\newblock Absence of diffusion in the {Anderson} tight binding model for large
  disorder or low energy.
\newblock {\em Commun. Math. Phys.}, 88:151--184, 1983.

\bibitem{GerminetK-01a}
F.~Germinet and A.~Klein.
\newblock Bootstrap multiscale analysis and localization in random media.
\newblock {\em Comm. Math. Phys.}, 222(2):415--448, 2001.

\bibitem{HislopK-02}
P.~D. Hislop and F.~Klopp.
\newblock The integrated density of states for some random operators with
  nonsign definite potentials.
\newblock {\em J. Funct. Anal.}, 195(1):12--47, 2002.

\bibitem{Kato-66}
T.~Kato.
\newblock {\em Perturbation Theory of Linear Operators}.
\newblock Springer, Berlin, 1966.

\bibitem{Kirsch-96}
W.~Kirsch.
\newblock {Wegner} estimates and {Anderson} localization for alloy-type
  potentials.
\newblock {\em Math. Z.}, 221:507--512, 1996.

\bibitem{Kirsch-08}
W.~Kirsch.
\newblock An invitation to random {S}chr\"odinger operators.
\newblock In {\em Random {S}chr\"odinger operators}, volume~25 of {\em Panor.
  Synth\`eses}, pages 1--119. Soc. Math. France, Paris, 2008.
\newblock With an appendix by Fr{\'e}d{\'e}ric Klopp,

\bibitem{KirschSS-98b}
W.~Kirsch, P.~Stollmann, and G.~Stolz.
\newblock Anderson localization for random {S}chr\"odinger operators with long
  range interactions.
\newblock {\em Comm. Math. Phys.}, 195(3):495--507, 1998.

\bibitem{Klopp-95a}
F.~Klopp.
\newblock Localization for some continuous random {Schr\"odinger} operators.
\newblock {\em Commun. Math. Phys.}, 167:553--569, 1995.

\bibitem{Klopp-02c}
F.~Klopp.
\newblock Weak disorder localization and {L}ifshitz tails: continuous
  {H}amiltonians.
\newblock {\em Ann. Henri Poincar\'e}, 3(4):711--737, 2002.

\bibitem{KloppN-09b}
F.~Klopp and S.~Nakamura.
\newblock Lifshitz tails for some non monotonous random models.
\newblock In {\em S\'eminaire: \'{E}quations aux {D}\'eriv\'ees {P}artielles.
  2007--2008}, S\'emin. \'Equ. D\'eriv. Partielles, pages Exp. No. XIV, 9.
  \'Ecole Polytech., Palaiseau, 2009.

\bibitem{KloppN-09a}
F.~Klopp and S.~Nakamura.
\newblock Spectral extrema and {L}ifshitz tails for non-monotonous alloy type
  models.
\newblock {\em Comm. Math. Phys.}, 287(3):1133--1143, 2009.

\bibitem{KloppNNN-03}
F.~Klopp, S.~Nakamura, F.~Nakano, and Y.~Nomura.
\newblock Anderson localization for 2{D} discrete {S}chr\"odinger operators
  with random magnetic fields.
\newblock {\em Ann. Henri Poincar\'e}, 4(4):795--811, 2003.

\bibitem{KostrykinV-06}
V.~Kostrykin and I.~Veseli\'c.
\newblock On the {Lipschitz} continuity of the integrated density of states for
  sign-indefinite potentials.
\newblock {\em Math. Z.}, 252(2):367--392, 2006.

\bibitem{Krueger}
H.~Kr\"uger.
\newblock  Localization for random operators with non-monotone potentials with exponentially decaying correlations.
\newblock See http://math.rice.edu/$\sim$hk7/papers.html.


\bibitem{LenzPPV-08}
D.~Lenz, N.~Peyerimhoff, O.~Post, and I.~Veseli\'c.
\newblock Continuity properties of the integrated density of states on
  manifolds.
\newblock {\em Jpn. J. Math.}, 3(1):121--161, 2008.

\bibitem{LottS-02}
J.~Lott and G.~Stolz.
\newblock The spectral minimum for random displacement models.
\newblock {\em J. Comput. Appl. Math.}, 148(1):133--146, 2002.

\bibitem{TautenhahnV-10}
M.~Tautenhahn and I.~Veseli\'c.
\newblock Spectral properties of discrete alloy-type models.
\newblock In {\em Proceedings of the XV th International Conference on
  Mathematical Physics, Prague, 2009}. p. 551--555, World Scientific, 2010.

\bibitem{Ueki-08}
N.~Ueki.
\newblock Wegner estimate and localization for random magnetic fields.
\newblock {\em Osaka J. Math.}, 45(3):565--608, 2008.

\bibitem{Veselic-circulant}
I.~Veseli\'c.
\newblock Wegner estimates for sign-changing single site potentials.
\newblock {http://arxiv.org/abs/0806.0482}.

\bibitem{Veselic-02a}
I.~Veseli\'c.
\newblock {W}egner estimate and the density of states of some indefinite alloy
  type {S}chr\"odinger operators.
\newblock {\em Lett. Math. Phys.}, 59(3):199--214, 2002.

\bibitem{Veselic-07b}
I.~Veseli\'c.
\newblock {\em {\it Existence and regularity properties of the integrated
  density of states of random {Schr\"odinger} Operators}}, volume Vol. 1917 of
  {\em Lecture Notes in Mathematics}.
\newblock Springer-Verlag, 2007.

\bibitem{Wegner-81}
F.~Wegner.
\newblock Bounds on the {DOS} in disordered systems.
\newblock {\em Z. Phys. B}, 44:9--15, 1981.

\end{thebibliography}

\def\cprime{$'$}\def\polhk#1{\setbox0=\hbox{#1}{\ooalign{\hidewidth
  \lower1.5ex\hbox{`}\hidewidth\crcr\unhbox0}}}

\end{document}